\documentclass[a4paper,pdftex,12pt]{article}

\usepackage{geometry}
\usepackage{times}
\usepackage{amssymb,amsmath}
\usepackage{epsfig, graphicx}

\newtheorem{Theorem}{Theorem}
\newtheorem{Corollary}[Theorem]{Corollary}
\newtheorem{Lemma}[Theorem]{Lemma}
\newtheorem{Definition}[Theorem]{Definition}
\newenvironment{proof}
{\begin{trivlist}\item[]{{\sc Proof.}}}{\hfill{$\square$}\noindent\end{trivlist}}

\title{Fast regocnition of planar non unit distance graphs}
\author{Sascha Kurz$^\star$}
\date{\small $^\star$Fakult\"at f\"ur Mathematik, Physik und Informatik, Universit\"at Bayreuth, Germany, sascha.kurz@uni-bayreuth.de}

\begin{document}

\maketitle

{\small
\noindent \textbf{Abstract:}
We study criteria attesting that a given graph can not be embedded in the plane so that neighboring vertices are at unit
distance apart and the straight line edges do not cross.
}

%\keywords{planar graphs, unit distances, quadrangulations, exhaustive enumeration, matchstick puzzle}
%\subjclass[2000]{05C10, 52A40}
%05C10 (Topological graph theory)
%52A40 (Geometric inequalities, etc. (convex geometry))
%\keywords[2000 \textit{Mathematics subject classification}]{05C10, 52A40}
%\keywords[\textit{Keywords and phrases}]{planar graphs, unit distances, quadrangulations, exhaustive enumeration, matchstick puzzle}

\section{Introduction}

\noindent
A graph $G=(V,E)$ is called a unit distance graph in the plane if there is an embedding of the vertex set $V$ into the plane such that every pair of adjacent vertices is at unit distance apart and the edges are non-crossing. Recognizing whether a given graph is a unit distance graph in the plane is NP-hard, see \cite{pre05493581,0744.05053}. Nevertheless the problem whether a given graph can be embedded into the plane with non-crossing straight-line edges having a prescribed Euclidean length occurs in several applications, see e.~g.{} \cite{pre05493581,1052.05004}. So in this paper we investigate necessary conditions for planar unit distance graphs which can be tested quickly. These objects can be visualized using matchsticks, so we also call them matchstick graphs. For a paper of similar flavor we refer the interested reader to \cite{pre05523164}, where a very effective and fast heuristic for the recognition of non-Yutsis graphs is presented. As a benchmark we consider $3$-connected $4$-regular planar graphs, where a graph $G$ is called $k$-connected if the graph remains connected after deleting fewer than $k$ vertices from the graph. These graphs can be generated very quickly using the computer program \texttt{plantri} \cite{plantri}\footnote{Actually, \texttt{plantri} has very fast routines for the generation of $3$-connected quadrangulations of the sphere, whose duals are $3$-connected $4$-regular planar graphs.} and their counts for $30$ to $33$ vertices are: $16387852863$, $59985464681$, $220320405895$, and $811796327750$.

%\begin{table}[htp]
%  \begin{center}
%    \caption{Number of $3$-connected planar $4$-regular graphs, see \cite{pre05382101}.}
%    \label{table_planar_quadrangulations}
%    \rule{\textwidth}{0.3mm}
%    \begin{tabular}{l|rrrrrrrrrrrrr}
%      \!\# vertices\!&\!6\!&\!8\!&\!9\!&\!10\!&\!11\!&\!12\!&\!13\!&\!14\!&\!15\!&\!16\!&\!17\!&\!18\!&\!19\!\\
%      \hline
%      \!\# graphs\!&\!1\!&\!1\!&\!1\!&\!3\!&\!3\!&\!11\!&\!18\!&\!58\!&\!139\!&\!451\!&\!1326\!&\!4461\!&\!14554\!
%    \end{tabular}
%    
%    \smallskip
%    
%    \begin{tabular}{l|rrrrrr}
%      \!\# vertices\!&\!20\!&\!21\!&\!22\!&\!23\!&\!24\!&\!25\!\\
%      \hline
%      \!\# graphs\!&\!49957\!&\!171159\!&\!598102\!&\!2098675\!&\!7437910\!&\!26490072\!
%    \end{tabular}
%    
%    \smallskip
%    
%    \begin{tabular}{l|rrrr}
%      \!\# vertices\!&\!26\!&\!27\!&\!28\!&\!29\!\\
%      \hline
%      \!\# graphs\!&\!94944685\!&\!341867921\!&\!1236864842\!&\!4493270976\!
%    \end{tabular}
%    
%    \smallskip
%    
%    \begin{tabular}{l|rrr}
%      \!\# vertices\!&\!30\!&\!31\!&\!32\!\\
%      \hline
%      \!\# graphs\!&\!16387852863\!&\!59985464681\!&\!220320405895\!
%    \end{tabular}
%    
%    \smallskip
%    
%    \begin{tabular}{l|rrrr}
%      \!\# vertices\!&\!33\!&\!34\!\\
%      \hline
%      \!\# graphs\!&\!811796327750\!&\!3000183106119\!
%    \end{tabular}
%    \rule{\textwidth}{0.3mm}
%  \end{center}
%\end{table}

\begin{Definition}
  Let $\mathcal{M}$ be a unit distance graph in the plane. By $n(\mathcal{M})=|V|$ we denote the
  number of vertices, by $\mathcal{K}(\mathcal{M})$ we denote the set of vertices which is situated
  on the outer face of $\mathcal{M}$ and by $\mathcal{I}(\mathcal{M})$ we denote the set of the remaining
  vertices. For the cardinality of $\mathcal{K}(\mathcal{M})$ we introduce the notation $k(\mathcal{M})$. By
  $A_i(\mathcal{M})$ we denote the number of faces of $\mathcal{M}$ which are $i$-gons. Here we also count
  the outer face.
\end{Definition}

Whenever it is clear from the context which matchstick graph $\mathcal{M}$ is meant we only write $n$, $\mathcal{K}$, $\mathcal{I}$, $k$, $A_i$ instead of $n(\mathcal{M})$, $\mathcal{K}(\mathcal{M})$, $\mathcal{I}(\mathcal{M})$, $k(\mathcal{M})$, $A_i(\mathcal{M})$.

One of the basic tools for planar graphs is the Eulerian polyhedron formula $|V|-|E|+|F|=2$ which leads to the following lemma:

\begin{Lemma}
  For a $r$-regular matchstick graph $\mathcal{M}$ we have
  \begin{eqnarray}
    |E|&=&\frac{1}{2}\cdot\sum_{i=3}^\infty i\cdot A_i,\nonumber\\
    |V|&=&\frac{1}{r}\cdot\sum_{i=3}^\infty i\cdot A_i,\nonumber\\
    |F|&=&\sum\limits_{i=3}^\infty A_i,\nonumber\\
    2r&=&\sum_{i=3}^\infty \!\left(i-\frac{r(i-2)}{2}\right) A_i\label{eq_A_i_sum}\\
    &=&A_3-A_5-2A_6-3A_7-4A_8-\dots,\nonumber\\
    n&=&|V|=|F|-2.\nonumber
  \end{eqnarray}
\end{Lemma}

\section{Necessary criteria of unit distance graphs in the plane}
\label{sec_criteria}

\noindent
As mentioned in the introduction we can not expect to find sufficient criteria for the existence of a unit distance embedding in the plane which can be checked algorithmically in polynomial time. Utilizing exact coordinates is not feasible in any case since some unit distance graphs may be flexible and thus have an uncountable number of embeddings in the plane. Even if a unique embedding of a unit distance graph in the plane exists there are some complications for computer algebra systems. In \cite{gerbracht} the exact coordinates of the Harborth graph, a $3$-connected $4$-regular planar graph consisting of $52$~vertices,  were explicitly calculated. Some minimal polynomials of vertex coordinates have a degree of $22$. The determination of the algebraic expressions for the coordinates of the vertices of the Harborth graph required an article of 18~pages and several days of cpu time. So there really is a need for coordinate-free criteria.

Our first argument uses the area of faces in planar graphs. The maximum area of an equilateral $k$-gon, whose edges have length $1$, is given by
\begin{equation}
  A_{\text{max}}(k)=\frac{k}{4}\cdot\cot\left(\frac{\pi}{k}\right).
\end{equation}
Using lower bounds for the area of the inner faces we can deduce some restrictions on parameters of unit distance graphs in the plane. Due to Equation~(\ref{eq_A_i_sum}) we have
$A_3\ge 4+k$ for $4$-regular matchstick graphs. In Table \ref{table_max_area} we have listed the maximum area of an equilateral $k$-gon measured in units of equilateral triangles.

\begin{table}[ht]
  \begin{center}
  \caption{Maximum number of equilateral triangles in an equil.{} $k$-gon.}
  \label{table_max_area}
  \rule{\textwidth}{0.3mm}
  \begin{tabular}{rr|rr|rr|rr|rr|rr}
    %\hline
    \!k\!&\!$\frac{A_{\text{max}}(k)}{A_{\text{max}}(3)}$\!&\!k\!&\!$\frac{A_{\text{max}}(k)}{A_{\text{max}}(3)}$\!&\!
    k\!&\!$\frac{A_{\text{max}}(k)}{A_{\text{max}}(3)}$\!&\!k\!&\!$\frac{A_{\text{max}}(k)}{A_{\text{max}}(3)}$\!&\!
    k\!&\!$\frac{A_{\text{max}}(k)}{A_{\text{max}}(3)}$\!&\!k\!&\!$\frac{A_{\text{max}}(k)}{A_{\text{max}}(3)}$\!\\[1.3mm]
    \hline
    \!3\!&\!1.00\!&\!4\!&\!2.31\!&\!5\!&\!3.98\!&\!6\!&\!6.00\!&\!7\!&\!8.40\!&\!8\!&\!11.16\!\\[1mm]
    \!9\!&\!14.28\!&\!10\!&\!17.77\!&\!11\!&\!21.63\!&\!12\!&\!25.86\!&\!13\!&\!30.46\!&\!14\!&\!35.42\!\\%[1mm]
    %\!15\!&\!40.744\!&\!16\!&\!46.441\!&\!17\!&\!52.506\!&\!18\!&\!58.938\!&\!19\!&\!65.738\!&\!20\!&\!72.905\!\\
    %\hline
  \end{tabular}
  \rule{\textwidth}{0.3mm}
  \end{center}
\end{table}

Concerning the minimum area of an equilateral $s$-gon in \cite{0968.51014} the authors have proved $A_{\text{min}}(s)=\frac{\sqrt{3}}{4}$ for odd $s$.
So an inner pentagon in a matchstick graph has an area of at least $\frac{\sqrt{3}}{4}$ and forces one additional triangle in Equation~(\ref{eq_A_i_sum}). An inner $s$-gon with $s\ge 6$ forces at least two additional triangles due to Equation~(\ref{eq_A_i_sum}). So for each given $k$ the number $A_i$ is bounded from above for $i\neq 4$. Only the quadrangles cause some trouble. Therefore we take a closer look at configurations of triangles and quadrangles at an inner point.

\begin{Lemma}
  \label{lemma_inner_3_4_configuration}
  If a $4$-regular matchstick graph contains an inner vertex but no inner $s$-faces with $s\ge 5$ then it
  contains one of the configurations of Figure \ref{fig_configuration_1}, where the total area of the quadrangles
  is greater than $\frac{\sqrt{3}}{2}$.
\end{Lemma}
\begin{proof}
  Due to the premises all inner faces are triangles and quadrangles. An inner vertex is part of four inner
  faces. Due to angle sums at most two of them can be triangles. If we denote the angles of the quadrangles
  at the central vertex by $\alpha_i$ the  area of the quadrangles is given by $\sum_i\left|\sin\alpha_i\right|$.
  A little calculus together with $0\le\alpha_i\le \pi$ yields the stated result.
\end{proof}

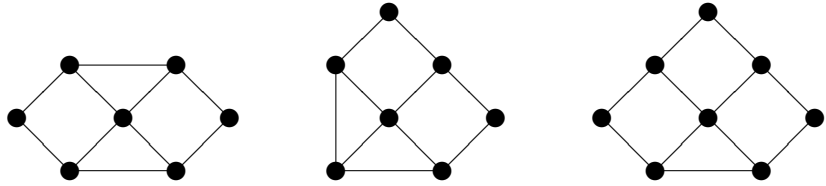
\begin{figure}[h]
  \begin{center}
    \setlength{\unitlength}{0.70cm}
    \begin{picture}(15,3)
      \put(1,0){\line(1,0){2}}
      \put(1,0){\line(1,1){1}}
      \put(3,0){\line(-1,1){1}}
      \put(2,1){\line(-1,1){1}}
      \put(2,1){\line(1,1){1}}
      \put(1,2){\line(1,0){2}}
      \put(1,0){\line(-1,1){1}}
      \put(0,1){\line(1,1){1}}
      \put(4,1){\line(-1,1){1}}
      \put(3,0){\line(1,1){1}}
      \put(1,0){\circle*{0.35}}
      \put(3,0){\circle*{0.35}}
      \put(1,2){\circle*{0.35}}
      \put(3,2){\circle*{0.35}}
      \put(2,1){\circle*{0.35}}
      \put(0,1){\circle*{0.35}}
      \put(4,1){\circle*{0.35}}
      \put(6,0){\line(0,1){2}}
      \put(6,0){\line(1,0){2}}
      \put(8,0){\line(-1,1){2}}
      \put(6,0){\line(1,1){1}}
      \put(8,0){\line(1,1){1}}
      \put(7,1){\line(1,1){1}}
      \put(6,2){\line(1,1){1}}
      \put(7,3){\line(1,-1){2}}
      \put(6,0){\circle*{0.35}}
      \put(6,2){\circle*{0.35}}
      \put(8,0){\circle*{0.35}}
      \put(7,1){\circle*{0.35}}
      \put(7,3){\circle*{0.35}}
      \put(8,2){\circle*{0.35}}
      \put(9,1){\circle*{0.35}}
      \put(12,0){\line(1,0){2}}
      \put(12,0){\line(-1,1){1}}
      \put(11,1){\line(1,1){1}}
      \put(14,0){\line(-1,1){2}}
      \put(12,0){\line(1,1){1}}
      \put(14,0){\line(1,1){1}}
      \put(13,1){\line(1,1){1}}
      \put(12,2){\line(1,1){1}}
      \put(13,3){\line(1,-1){2}}
      \put(11,1){\circle*{0.35}}
      \put(12,0){\circle*{0.35}}
      \put(12,2){\circle*{0.35}}
      \put(14,0){\circle*{0.35}}
      \put(13,1){\circle*{0.35}}
      \put(13,3){\circle*{0.35}}
      \put(14,2){\circle*{0.35}}
      \put(15,1){\circle*{0.35}}
    \end{picture}\\[2mm]
    \caption{Configurations of triangles and quadrangles.}
    \label{fig_configuration_1}
  \end{center}
\end{figure}

Determining the numbers $A_i$ for odd $i$, searching for a configuration from Figure~\ref{fig_configuration_1} and comparing the resulting lower bound on the area of the inner faces with the maximum values in Table~\ref{table_max_area} can clearly be done in linear time. We will call this the area argument.

\begin{figure}[h]
  \begin{center}
    \setlength{\unitlength}{0.70cm}
    \begin{picture}(4,4.5)
      \put(1,1){\line(1,0){2}}
      \put(0,2){\line(1,0){4}}
      \put(1,3){\line(1,0){2}}
      \put(1,1){\line(0,1){2}}
      \put(2,0){\line(0,1){4}}
      \put(3,1){\line(0,1){2}}
      \put(0,2){\line(1,1){2}}
      \put(0,2){\line(1,-1){2}}
      \put(4,2){\line(-1,-1){2}}
      \put(4,2){\line(-1,1){2}}
      \put(1,1){\circle*{0.35}}
      \put(1,2){\circle*{0.35}}
      \put(1,3){\circle*{0.35}}
      \put(2,1){\circle*{0.35}}
      \put(2,2){\circle*{0.35}}
      \put(2,3){\circle*{0.35}}
      \put(3,1){\circle*{0.35}}
      \put(3,2){\circle*{0.35}}
      \put(3,3){\circle*{0.35}}
      \put(0,2){\circle*{0.35}}
      \put(2,0){\circle*{0.35}}
      \put(4,2){\circle*{0.35}}
      \put(2,4){\circle*{0.35}}
    \end{picture}
    \caption{Example of a non unit distance graph.}
    \label{fig_example_1}
  \end{center}
\end{figure}
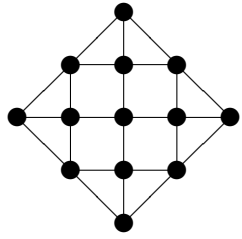

In order to enhance the area argument we may try to locate several face disjoint subgraphs from Figure~\ref{fig_configuration_1}. Therefore we consider the following binary linear program (BLP) which determines the maximum number of such face disjoint subgraphs for a given matchstick graph. For a precise formulation we need some further notation. We consider the faces which are adjacent to a given vertex $v$. By $fn(v)$ we denote the set of vertices and by $fs(v)$ we denote the multiset of the cardinalities of these faces. For the central vertices of the matchstick graphs in Figure ~\ref{fig_configuration_1} the face sets $fs(v)$ are given by $\{3,3,4,4\}$, $\{3,4,4,4\}$, and $\{3,4,4,4\}$, respectively.

\begin{equation}
  \max \sum_{v\in V} x_v\label{lp_3_4_configurations}
\end{equation}
subject to
\begin{align}
x_v & =0 && \forall v\in V,\,fs(v)\neq \{3,3,4,4\},\{3,4,4,4\}\nonumber\\
x_v+x_u & \le 1 &&\forall v\in V,\, u\in fn(v)\nonumber\\%\label{ie_disjoint}\\ 
x_v & \geq 0 && \forall v\in V\nonumber\\
x_v & \in\{0,1\} && \forall v\in V\nonumber
\end{align}

An example where the non-existence of a unit distance embedding can be shown using the area argument and BLP~(\ref{lp_3_4_configurations}) is given in Figure~\ref{fig_example_1}.

In addition to arguments based on area estimates we can also utilize information on the perimeter of sub-configurations to deduce the impossibility of a unit distance embedding in some cases.
\begin{Lemma}
  \label{lemma_perimeter}
  If $P$ is a polygon inside another polygon $Q$, where $P\neq Q$, then the perimeter of $Q$ has to be strictly larger than
  the perimeter of the convex hull of $P$.
\end{Lemma}
So we have a proof for the obvious fact that no equilateral quadrangle can contain another equilateral quadrangle of the same edge length. This argument is more valuable in combination with triangle chains, see Figure~\ref{fig_triangle_chain}.

\begin{figure}[h]
  \begin{center}
    \setlength{\unitlength}{0.70cm}
    \begin{picture}(6,2.5)
      \put(0,0.5){\line(1,0){6}}
      \put(0,0.5){\line(2,3){1}}
      \put(2,0.5){\line(2,3){1}}
      \put(4,0.5){\line(2,3){1}}
      \put(1,2){\line(2,-3){1}}
      \put(3,2){\line(2,-3){1}}
      \put(5,2){\line(2,-3){1}}
      \put(1,2){\line(1,0){4}}
      \put(0,0.5){\circle*{0.35}}
      \put(2,0.5){\circle*{0.35}}
      \put(4,0.5){\circle*{0.35}}
      \put(6,0.5){\circle*{0.35}}
      \put(1,2){\circle*{0.35}}
      \put(3,2){\circle*{0.35}}
      \put(5,2){\circle*{0.35}}
      \put(0,0){$x_0$}
      \put(2,0){$x_1$}
      \put(3.8,0){$\dots$}
      \put(6,0){$x_s$}
    \end{picture}
    \caption{Triangle chains.}
    \label{fig_triangle_chain}
  \end{center}
\end{figure}
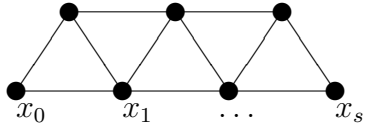

\begin{Corollary}
  \label{cor_triangle_chain}
  If a matchstick graph $\mathcal{M}$ contains a triangle chain with vertices $x_0,\dots,x_s$, then it
  is either a triangulation or we have $k(\mathcal{M})\ge 2s+2$.
\end{Corollary}

We remark that triangle chains which are maximal with respect to inclusion can be determined in linear time. 

Another source for conditions on unit distance graphs in the plane comes from angle arguments for $2$-connected planar graphs. The first easy facts are: The inner angles of an equilateral triangle equal $\frac{\pi}{3}$, the sum of the inner angles of an $s$-gon equals $(s-2)\pi$ for inner faces and $(s+2)\pi$ for the outer face. Two neighboring inner angles of an equilateral quadrangle sum up to $\pi$. For inner $s$-gons with $s\ge 5$ we have:

\begin{Lemma}
  \label{lemma_angles}
  If $\alpha$ and $\beta$ are two neighboring angles of an equilateral $k$-gon, then $\alpha+\beta>\frac{\pi}{2}$ holds.
\end{Lemma}
\begin{proof}
  We assume w.l.o.g.{} $\alpha\le\beta$. For minimal $\alpha+\beta$ we can assume that the two arms of the neighboring angles touch,
  so that we have an isosceles triangle with $\alpha+2\beta=\pi$. Thus $\beta<\frac{\pi}{2}$ and
  $\alpha+\beta=\pi-\beta>\frac{\pi}{2}$.
\end{proof}

With this we can check in linear time if the resulting bounds for the sum of the inner or outer angles contradicts the known exact value for each face of a given matchstick graph. For the planar graph on the left hand side in Figure~\ref{fig_example_2} we know that the outer angles of the central quadrangle should sum up to $6\pi$. On the other hand this quadrangle is adjacent to $8$~angles which are part of a triangle and to two pairs of neighbored angles which are part of quadrangles. Due to $8\cdot\frac{\pi}{3}+2\cdot\pi\neq (4+2)\cdot\pi$ we have a contradiction. For the planar graph on the right hand side of Figure~\ref{fig_example_2} we obtain a contradiction by looking at the angles of the outer heptagon. We will call this procedure the angle argument.

%know the inner angles of the outer heptagon should sum up to $(7-2)\cdot\pi$. On the other hand the outer face is adjacent to $14$ angles which are part of triangles, $1$ angle which is part of a pentagon, and $1$ pair of neighbored angles of a quadrangle. Due to $14\cdot\frac{\pi}{3}+1\cdot\varepsilon+1\cdot\pi>(7-2)\cdot \pi$ for all $\varepsilon>0$ we have a contradiction. 

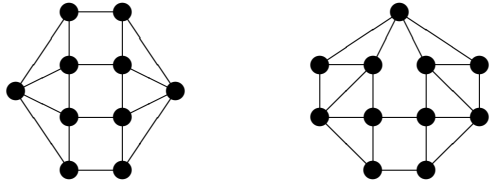
\begin{figure}[h]
  \begin{center}
    \setlength{\unitlength}{0.70cm}
    \begin{picture}(3,3.5)
      \put(1,0){\line(1,0){1}}
      \put(1,1){\line(1,0){1}}
      \put(1,2){\line(1,0){1}}
      \put(1,3){\line(1,0){1}}
      \put(1,0){\line(0,1){3}}
      \put(2,0){\line(0,1){3}}
      \put(0,1.5){\line(2,1){1}}
      \put(0,1.5){\line(2,-1){1}}
      \put(0,1.5){\line(2,3){1}}
      \put(0,1.5){\line(2,-3){1}}
      \put(3,1.5){\line(-2,1){1}}
      \put(3,1.5){\line(-2,-1){1}}
      \put(3,1.5){\line(-2,3){1}}
      \put(3,1.5){\line(-2,-3){1}}
      \put(1,0){\circle*{0.35}}
      \put(1,1){\circle*{0.35}}
      \put(1,2){\circle*{0.35}}
      \put(1,3){\circle*{0.35}}
      \put(2,0){\circle*{0.35}}
      \put(2,1){\circle*{0.35}}
      \put(2,2){\circle*{0.35}}
      \put(2,3){\circle*{0.35}}
      \put(0,1.5){\circle*{0.35}}
      \put(3,1.5){\circle*{0.35}}
    \end{picture}
    \quad\quad\quad\quad
    \setlength{\unitlength}{0.70cm}
    \begin{picture}(3,3.5)
      \put(1,0){\line(1,0){1}}
      \put(0,1){\line(1,0){3}}
      \put(0,2){\line(1,0){1}}
      \put(2,2){\line(1,0){1}}
      \put(1,0){\line(0,1){2}}
      \put(2,0){\line(0,1){2}}
      \put(0,1){\line(0,1){1}}
      \put(3,1){\line(0,1){1}}
      \put(0,1){\line(1,-1){1}}
      \put(0,1){\line(1,1){1}}
      \put(3,1){\line(-1,-1){1}}
      \put(3,1){\line(-1,1){1}}
      \put(1,2){\line(1,2){0.5}}
      \put(2,2){\line(-1,2){0.5}}
      \put(0,2){\line(3,2){1.5}}
      \put(3,2){\line(-3,2){1.5}}
      \put(1,0){\circle*{0.35}}
      \put(2,0){\circle*{0.35}}
      \put(0,1){\circle*{0.35}}
      \put(1,1){\circle*{0.35}}
      \put(2,1){\circle*{0.35}}
      \put(3,1){\circle*{0.35}}
      \put(0,2){\circle*{0.35}}
      \put(1,2){\circle*{0.35}}
      \put(2,2){\circle*{0.35}}
      \put(3,2){\circle*{0.35}}
      \put(1.5,3){\circle*{0.35}}
    \end{picture}
    \caption{Examples of non unit distance graphs.}
    \label{fig_example_2}
  \end{center}
\end{figure}

We would like to remark that the \textit{local} angle argument does not give any contradiction for the planar graph in Figure~\ref{fig_example_3}. If we incorporate all our \textit{local} conditions into the following linear program, we can deduce a contradiction. Therefore let $\mathcal{A}$ the set of angles in a given $2$-connected planar graph $\mathcal{M}$ and $o(f)\subseteq\mathcal{A}$ the set of outer angles for a given face $f$ of $\mathcal{M}$ consisting of $|f|$ edges.

\begin{equation}
  \max y\label{angle_lp}
\end{equation}
subject to
\begin{align}
x_a -y & \ge 0 && \forall a\in\mathcal{A}\nonumber\\
x_a & = \frac{\pi}{3} && \forall a\in\mathcal{A}\text{ which are part of a triangle}\nonumber\\
x_u+x_v & = \pi && \forall u,v\in\mathcal{A}\text{ which are adjacent angles of}\nonumber\\
&&&\text{a quadrangle}\nonumber\\
x_u+x_v-y & \ge \pi && \forall u,v\in\mathcal{A}\text{ which are adjacent angles of}\nonumber\\
&&&\text{a $s$-gon with $s\ge 5$}\nonumber\\
\sum_{a\in o(f)} x_a & = (|f|+2)\cdot \pi && \forall \text{ inner faces $f$ of $\mathcal{M}$}\nonumber\\
\sum_{a\in o(f)} x_a & = (|f|-2)\cdot \pi && \text{for the outer face $f$ of $\mathcal{M}$}\nonumber
\end{align}

If the linear program~(\ref{angle_lp}) does not have a feasible solution with $y>0$ then $\mathcal{M}$ cannot be a matchstick graph.

\begin{figure}[h]
  \begin{center}
    \setlength{\unitlength}{0.70cm}
    \begin{picture}(6,3.5)
      \put(0,1){\line(1,0){5}}
      \put(0,2){\line(1,0){5}}
      \put(0,1){\line(0,1){1}}
      \put(3,0){\line(0,1){3}}
      \put(4,1){\line(0,1){1}}
      \put(5,1){\line(0,1){1}}
      \put(4,2){\line(-1,1){1}}
      \put(5,2){\line(-2,1){2}}
      \put(4,1){\line(-1,-1){1}}
      \put(5,1){\line(-2,-1){2}}
      \put(2,2){\line(1,1){1}}
      \put(2,1){\line(1,-1){1}}
      \put(6,1.5){\line(-2,1){1}}
      \put(6,1.5){\line(-2,-1){1}}
      \put(1,1.5){\line(-2,-1){1}}
      \put(1,1.5){\line(-2,1){1}}
      \put(1,1.5){\line(2,-1){1}}
      \put(1,1.5){\line(2,1){1}}
      \put(0,1){\circle*{0.35}}
      \put(0,2){\circle*{0.35}}
      \put(1,1.5){\circle*{0.35}}
      \put(2,1){\circle*{0.35}}
      \put(2,2){\circle*{0.35}}
      \put(3,0){\circle*{0.35}}
      \put(3,1){\circle*{0.35}}
      \put(3,2){\circle*{0.35}}
      \put(3,3){\circle*{0.35}}
      \put(4,1){\circle*{0.35}}
      \put(4,2){\circle*{0.35}}
      \put(5,1){\circle*{0.35}}
      \put(5,2){\circle*{0.35}}
      \put(6,1.5){\circle*{0.35}}
    \end{picture}
    \caption{Example of a non unit distance graph.}
    \label{fig_example_3}
  \end{center}
\end{figure}
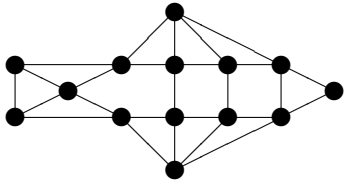

% computation times for the plantri-plugin:
% n=22: 2.6s
% n=23: 8.8s
% n=24: 30.8s
% n=25: 1m48s
% n=26: 6m22s
% n=27: 22m37s
% n=28: 80m57s
% n=29: 
% n=30: <1070m
% n=31: ~2300m
% n=32: >4000m
% n=33:

Applying the checks described above which can be performed in linear time leaves over only $5$, $50$, $279$, and $5051$ graphs for $30$, $31$, $32$, and $33$ vertices, respectively. In all these cases the corresponding linear programs~(\ref{angle_lp}) do not have feasible solutions, so that we have $n\ge 34$ for each $3$-connected $4$-regular matchstick graph.

\section{Conclusion and outlook}

\noindent
In this paper we have developed some simple and algorithmically fast necessary criteria for unit distance graphs in the plane. As a benchmark problem we have used the $3$-connected $4$-regular planar graphs with up to $33$~vertices, where the necessary criteria were also sufficient.

We would like to end with some open problems. The key motivation for this article was the question for the smallest $4$-regular matchstick graph, see also \cite{0756.05068}. To this end we will prove a lower bound of $34$~vertices in a companion paper. Similarly one can ask for the smallest $3$-regular matchstick graph with girth~$5$. Here the smallest known example consists of $180$~vertices \cite{girth_four}. The algorithmically most challenging problem is an exhaustive enumeration of all matchstick graphs for a reasonable number of vertices. The known bounds for the maximum number $\tilde{u}(n)$ of edges, see e.~g.{} \cite{1086.52001}, are given by
\[
  \left\lfloor 3n-\sqrt{12n-3}\right\rfloor\le \tilde{u}(n)\le 3n-O\!\left(\sqrt{n}\right),
\]
where the lower bound is conjectured to be exact. If the edges are allowed to cross, we refer to \cite{schade} for the exact values of the maximum number of edges for $n\le 14$. Of course we have to mention the famous open question for the chromatic number of unit distance graphs (or the plane).

\nocite{matchsticks_in_the_plane}

%% \bibliographystyle{amsplain}
%% \bibliography{unmasking_non_unit_distance_graphs}

\end{document}